\documentclass[letterpaper,10pt]{article}
\usepackage{amsmath, amsthm, amssymb}

\theoremstyle{plain}
\newtheorem{thm}{Theorem}
\newtheorem*{thm*}{Theorem}
\newtheorem{lemma}[thm]{Lemma}

\newtheorem*{prop*}{Proposition}

\newtheorem*{question*}{Question}

\theoremstyle{definition}
\newtheorem*{definition*}{Definition}
\newtheorem*{example*}{Example}
\newtheorem{remark}[thm]{Remark}
\newtheorem*{remark*}{Remark}

\renewcommand{\leq}{\leqslant}
\renewcommand{\geq}{\geqslant}

\newcommand{\core}[1]{\operatorname{core}(#1)}
\newcommand{\soc}[1]{\operatorname{socle}(#1)}
\newcommand{\tor}{\operatorname{Tor}}
\newcommand{\hf}[2]{\operatorname{HF}\!\left(#1,\,#2\right)}

\newcommand{\m}{\mathfrak{m}}
\newcommand{\cm}{Cohen-Macaulay}
\renewcommand{\l}{\ensuremath{\underline \ell}}
\newcommand{\f}{\ensuremath{\underline f}}
\newcommand{\fN}{\ensuremath{f_{N+1}}}
\newcommand{\bfN}{\ensuremath{\bar f_{N+1}}}
\newcommand{\dN}{\ensuremath{d_{N+1}}}

\title{\bf Bound on the multiplicity of\\ 
		   almost complete intersections}
\author{Bahman Engheta
		\\
		\\
		\small\it
		Department of Mathematics, University of California, Riverside, CA 92521
		\\
		\small\it
		E-mail address: \rm engheta@math.ucr.edu}
\date{}

\begin{document}

\maketitle 

\begin{abstract}
Let $R$ be a polynomial ring over a field of characteristic zero and let $I\subset R$ be a graded ideal of height $N$ which is minimally generated by $N+1$ homogeneous polynomials. 
If $I=(f_1,\ldots,\fN)$ where $f_i$ has degree $d_i$ and $(f_1,\ldots,f_N)$ has height $N$, then the multiplicity of $R/I$ is bounded above by $\prod_{i=1}^N d_i - \max\left\{ 1,\ \sum_{i=1}^N (d_i-1) - (\dN-1) \right\}$. 
\end{abstract}


\section{Introduction}

Let $R=k[X_1,\ldots,X_n]$ where $k$ is a field of characteristic zero and let $f_1,\ldots,f_N$ be forms in $R$ of degrees $d_1,\ldots,d_N$, respectively, which generate an ideal of height $N$. 
That is, $f_1,\ldots,f_N$ form a regular sequence and they generate a complete intersection ideal with multiplicity $\prod_{i=1}^N d_i$. 
Let $\fN$ be yet another form, of degree $\dN$, which is a zero-divisor on $R/(f_1,\ldots,f_N)$, that is, $(f_1,\ldots,f_N) \subsetneq (f_1,\ldots,f_N):\fN \subsetneq R$. 
Herein the ideal $(f_1,\ldots,\fN)$ is referred to as an \emph{almost complete intersection}. 
Throughout this article, let $I=(f_1,\ldots,\fN)$ and $\f=f_1,\ldots,f_N$. 
We will abuse the notation $\f$ to also denote the ideal generated by this sequence. 

It can be shown, as one would expect, that the multiplicity of $R/I$ is strictly less than that of $R/\f$. 
This note is aimed at making this fact more precise and exhibiting a bound for the multiplicity of $R/I$ in terms of the degrees $d_1,\ldots,\dN$ of the minimal generators of $I$. 
Our approach leads quite naturally to the notion of the core of an ideal, which was first alluded to by Rees and Sally in \cite{rs} and treated more explicitly by Huneke and Swanson in \cite{hs}, and which has been the subject of growing interest in recent years. 

\smallskip%

{\it The core of an ideal.} 
Recall that given two ideals $B\subseteq A$, $B$ is said to be a \emph{reduction} of $A$ if $A^{t+1} = BA^t$ for some non-negative integer $t$, and consequently for all integers greater than $t$. 
(Equivalently, a reduction of $A$ is a subideal of $A$ with the same integral closure as $A$.) 
The least such integer is called the \emph{reduction number} of $A$ with respect to $B$. 
A reduction $B$ of $A$ is called \emph{minimal} if no ideal properly contained in $B$ is a reduction of $A$, that is, if it is minimal with respect to inclusion. 
As minimal reductions are not unique, one is led to consider the intersection of all minimal reductions: the \emph{core} of an ideal $A$, denoted $\core A$, is defined as the intersection of all reductions (equivalently, minimal reductions) of $A$. 

\smallskip%

We first take advantage of the well-understood structure of $R/\f$ and express its multiplicity $\prod_{i=1}^N d_i$ as the length of $R/(\f,\l)$, where $\l=\ell_1,\ldots,\ell_{n-N}$ is a sequence of $n-N$ general linear forms. 
Our focus will then be to determine when, and by how much, the additional generator $\fN$ further reduces the length of $R/(\f,\l)$ to that of $R/(I,\l)$. 
Finally, as the multiplicity of $R/I$ is no greater than the length of $R/(I,\l)$, we obtain 
\begin{equation} \label{ineq}
	e(R/I) 
	\ \leq \ 
	\lambda( R/{(I,\l)} ) 
	\ \lneq \ 
	\lambda( R/{(\f,\l)} ) 
	\ = \ 
	e(R/\f), 
\end{equation}
where $\lambda(~)$ and $e(~)$ denote length and multiplicity, respectively.
More concretely, we prove the following 
\begin{thm} \label{main}
Let $R$ be a polynomial ring over a field of characteristic 0. 
If $I \subset R$ is an almost complete intersection minimally generated by $f_1,\ldots,\fN$ such that $f_1,\ldots,f_N$ form a regular sequence, then the multiplicity of $R/I$ is at most 
\[
	\prod_{i=1}^N d_i - \max\left\{ 1,\ \sum_{i=1}^N (d_i-1) - (\dN-1) \right\}, 
\]
where $d_i = \deg(f_i)$ for $i=1\ldots N$.
\end{thm}

Let $\bar R:=R/\f$ and let $\bar~$ denote the residue class in $\bar R$. 
Our argument is based on the well-known fact (see \cite{nr}) that in $\bar R$ any choice of $n-N$ general linear forms $\l$ generates a graded minimal reduction of the homogeneous maximal ideal $\bar\m=(\bar X_1,\ldots,\bar X_n)$, and all graded minimal reductions of $\bar\m$ are of this form. 

Note that the strict inequality in \eqref{ineq} holds if and only if $\fN \notin (\f,\l)$. 
Thus, to establish this inequality it suffices to show that the image of $\fN$ in $\bar R$ is not contained in some graded minimal reduction of $\bar\m$. 
To this end, we appeal to a result of Corso, Polini, and Ulrich \cite[Theorem 4.5]{cpu1} which implies that in our setting $\core{\bar\m}$ is in fact the intersection of the graded minimal reductions of $\bar\m$, and we show that $\bfN \notin \core{\bar\m}$. 
To compute $\core{\bar\m}$ we avail ourselves of a formula that was conjectured by Corso, Polini, and Ulrich in \cite{cpu2}, and was proved independently by Huneke and Trung \cite[Theorem 3.7]{ht} and Polini and Ulrich \cite[Theorem 4.5]{pu}.

\subsection{Preliminaries} \label{prelims}

We recall that given a sequence $\l=\ell_1,\ldots,\ell_{n-N}$ of general linear forms, the multiplicity of $R/I$ is bounded above by the length of $R/(I,\l)$. 
Indeed, as $I$ has height $N$, the elements $\l$ constitute a system of parameters of $R/I$. 
The multiplicity of $R/I$ can be obtained as the Euler characteristic of $R/I$ with respect to $\l$
\[ \begin{split}
	\chi( \l,R/I )	\ & = \ \sum_{i\geq0} (-1)^i \, \lambda\big( H_i(\l,R/I) \big) \\
					\ & = \ \sum_{i\geq0} (-1)^{i} \, \lambda\big( \tor_i^R(R/\l,R/I) \big), 
\end{split} \]
where $H_{\bullet}(\l,R/I)$ denotes the Koszul homology of $\l$ with coefficients in $R/I$. 
If we further consider the first partial Euler characteristic 
\[
	\chi_{_1}( \l,R/I ) \ = \ \sum_{i \geq 1} (-1)^{i-1} \, \lambda\big( H_{i}(\l,R/I) \big), 
\]
then we have $\chi( \l,R/I ) = \lambda(R/(I,\l)) - \chi_{_1}( \l,R/I )$ and the following non-negativity result yields $e(R/I) \leq \lambda(R/(I,\l))$, as desired. 
\begin{thm*}[Serre \cite{s}]
The first partial Euler characteristic $\chi_{_1}( \l,R/I )$ is non-negative, or equivalently, $\chi( \l,R/I ) \leq \lambda(R/(I,\l))$.
\end{thm*}

\medskip%

{\it Socle degree of complete intersections.} 
Recall that in our setting $\f$ is a regular sequence $f_1,\ldots,f_N$  with $\deg(f_i)=d_i$ and $\l=\ell_1,\ldots,\ell_{n-N}$ is a sequence of general linear forms. 
Next we point out that $R/(\f,\l)$ has a pure socle generated in degree $\sum_{i=1}^N (d_i-1)$. 
To see this, we compute $\tor_n^R(R/(\f,\l), k)$ twice. 
On the one hand, resolving $k$ via the Koszul complex on $X_1,\ldots,X_n$ and tensoring with $R/(\f,\l)$ yields 
\[
	\tor_n^R\!\left(R/(\f,\l), k\right) \ \cong \ \tfrac{(\f,\l):\m}{(\f,\l)}(-n).
\]
On the other hand, resolving $R/(\f,\l)$ via the Koszul complex on $\f,\l$ and tensoring with $R/\m$ we have  
\[
	\tor_n^R\!\left(R/(\f,\l), k\right) \ \cong \ k\big( \!-\!( n-N+\sum_{i=1}^N d_i ) \big).
\]
Hence $\soc{R/(\f,\l)} \cong k\big( \!-\!( \sum_{i=1}^N d_i-N ) \big)$, as claimed. 
Also note that $\lambda(R/(\f,\l)) = e(R/\f) = \prod_{i=1}^N d_i$. 

\begin{remark} \label{red-num}
For brevity set $r := \sum_{i=1}^N (d_i-1)$. 
It follows that $\m^{r+1} \subseteq (\f,\l)$, or equivalently, $\m^{r+1} \equiv \l \, \m^r$ modulo $\f$. 
Thus, in $\bar R=R/\f$ the minimal reduction $\l$ of $\m$ has reduction number $\leq r$. 
We also note that $\bar \m \subset \bar R$ is an \emph{equimultiple} ideal, that is, its height equals its analytic spread $n-N$. 
\end{remark}

\section{Computation of the core}

The following theorem provides a formula for the core of an equimultiple ideal in a \cm{} local ring.

\begin{thm}
[Huneke-Trung {\cite[3.7]{ht}}, Polini-Ulrich {\cite[4.5]{pu}}] \label{core-formula}
Let $S$ be a \cm{} local ring with residue field of characteristic $0$. 
Let $A$ be an equimultiple ideal of $S$ and let $B\subseteq A$ be a minimal reduction of $A$ with reduction number $r$. 
Then $\core{A} = B^{r+1}:A^r$. 
\end{thm}

Note that if $A$ is an equimultiple ideal, then its minimal reduction $B$ is generated by a regular sequence and it is easily seen that \mbox{$B^{r+1} : A^r$} $=$ \mbox{$B^{t+1} : A^t$} for all $t \geq r$. 
In \cite[Proposition 2.1]{puv} it is further shown that forming the core of zero-dimensional ideals commutes with localization. 
Thus, the colon formula of Theorem~\ref{core-formula} may be applied in our setting (see also \cite[Theorem 2.3]{puv}) and it yields 
\begin{equation} \label{colon-formula}
	\core{\bar\m} = \frac{(\l^{r+1},\,\f) \,:\, \m^r}{\f}, 
\end{equation}
where $r = \sum_{i=1}^N (d_i-1)$ as set previously. 
Using \eqref{colon-formula} and minimal free resolutions, we now describe $\core{\bar\m}$ more precisely with the following 
\begin{lemma} \label{core-m}
Let $R$ be a polynomial ring over a field of characteristic zero and let $\f \subset R$ be an ideal generated by a regular sequence of $N$ forms of degrees $d_1,\ldots,d_N$. 
Then $\core{\bar\m} = \bar\m^{r+1}$ with $r = \sum_{i=1}^N (d_i-1)$, where $\bar~$ denotes the residue class in $R/\f$. 
\end{lemma}
\begin{proof}
Say $R=k[X_1,\ldots,X_n]$ and $\f$ is generated by the regular sequence $f_1,\ldots,f_N$ with $\deg(f_i)=d_i$. 
Set $\bar R:=R/\f$ and let $\l=\ell_1,\ldots,\ell_{n-N}$ be a sequence of general linear forms in $R$. 
By Remark~\ref{red-num}, the image of $\l$ in $\bar R$ generates a minimal reduction of $\bar\m$ with reduction number $\leq r$ and by \eqref{colon-formula}, \mbox{$\core{\bar\m} = \frac{(\l^{r+1},\,\f) \,:\, \m^r}{\f}$}. 
To compute this colon, we first resolve $R/(\l^{r+1},\f)$ to determine its socle degree. 
Let 
\[
	\phi := 
	\left( \begin{matrix}\\ \\ \\ \\ \\ \end{matrix} \right.
	\underbrace{ \begin{array}{rrrrrrr}
		\ell_1	&\cdots	&\ell_{n-N}	&		&		&		& \\
			&\ell_1		&\cdots	&\ell_{n-N}	&		&		& \\
			&		&\ddots	&		&\ddots	&		& \\
			&		&		&\ell_1		&\cdots	&\ell_{n-N}	& \\
			&		&		&		&\ell_1		&\cdots	&\ell_{n-N}
	\end{array} }_{n-N+r \mathrm{~columns}}
	\left. \begin{matrix}\\ \\ \\ \\ \\ \end{matrix} \right)
	\begin{matrix}
		\left. \begin{matrix}\\ \\ \\ \\ \\ \end{matrix} \right\} 
		r+1 \mathrm{~rows}
	\end{matrix}
\] 
and note that the $(r+1)\times(r+1)$ minors of $\phi$ generate the ideal $\l^{r+1}$. 
Recall that $R/\l^{r+1}$ is perfect of grade $n-N$ and is minimally resolved by the Eagon-Northcott complex -- see \cite{en}: 
\begin{multline*}
	\mathbb{EN}(\phi): \quad 0
	\to R^{b_{n-N}} (-(n-N+r)) \xrightarrow{d_{n-N}} 
	\cdots
	\\
	\cdots \to R^{b_2} (-(r+2))
	\xrightarrow{d_2} R^{b_1} (-(r+1))
	\xrightarrow{\wedge^{^{\!r+1}}\phi} R
	\to R/\l^{r+1}
	\to 0.
\end{multline*}
On the other hand, $R/\f$ is minimally resolved by the Koszul complex: 
\[
	\mathbb{K}(\f) : \quad 0
	\rightarrow R(-\sum_{i=1}^N d_i)
	\rightarrow \cdots
	\rightarrow \bigoplus_{i=1}^N R(-d_i)
	\rightarrow R
	\rightarrow R/\f
	\rightarrow 0.
\]
We consider the tensor product of the above complexes 
\begin{multline*}
	\mathbb{EN}(\phi) \otimes \mathbb{K}(\f): \quad 0
	\to R^{b_{n-N}} ( -(n-N+r+\sum_{i=1}^N d_i) )
	\to \cdots
	\\
	\cdots \to
	R^{b_1} (-(r+1)) \oplus \bigoplus_{i=1}^N R(-d_i)
	\to R
	\to R/(\l^{r+1},\f)
	\to 0,
\end{multline*}
and recall that its $i$-th homology is isomorphic to $\tor_i^R(R/\l^{r+1},R/\f)$ for $i \geq 0$ -- see \cite[Theorem 11.21]{rot}. 
As $\l$ is generated by a regular sequence and $\f$ is a regular sequence modulo $\underline \ell$, it is also a regular sequence modulo $\l^{r+1}$ and all higher $\operatorname{Tor}_i^{}$ vanish. 
Thus, $\mathbb{EN}(\phi) \otimes \mathbb{K}(\f)$ is in fact a free resolution of $R/(\l^{r+1},\f)$ of length $n$ in which the $n$-th module has a twist of $-(n+r+\sum_{i=1}^N(d_i-1)) = -(n+2r)$. 
It now follows from 
\[
	\tor_n^R(R/(\l^{r+1},\f),k) ~\cong~ k^{b_{n-N}}(-(n+2r)) ~\cong~ \frac{(\l^{r+1},\f) : \m}{(\l^{r+1},\f)}(-n).
\]
that $R/(\l^{r+1},\f)$ has socle isomorphic to $k^{b_{n-N}}(-2r)$. 

To prove our claim, let $x \in R$ such that $\bar x \in \core{\bar\m}$. 
As $\core{\bar\m} = \frac{(\l^{r+1},\,\f) \,:\, \m^r}{\f}$, we have $x\,\m^r \subseteq (\l^{r+1},\,\f)$ and consequently $x\,\m^{r-1}$ is contained in the socle of $R/(\l^{r+1},\f)$. 
This socle is generated in degree $2r$, as shown above. 
Thus, $\deg x \geq r+1$ and $\core{\bar\m} \subseteq \bar\m^{r+1}$. 
The reverse inclusion is clear, as $\m^{2r+1} \subseteq (\l^{r+1},\f)$. 
\end{proof}

\medskip%

We are now ready to prove Theorem~\ref{main}. 

\begin{proof}[Proof of Theorem~\ref{main}]
As before, let $\bar~$ denote the residue class in $\bar R=R/\f$ and let $\m$ be the homogeneous maximal ideal $(X_1,\ldots,X_n)$. 
By Lemma~\ref{core-m}, $\core{\bar\m}$ is generated in degree $r+1$, where $r = \sum_{i=1}^N (d_i-1)$. 
So if $\deg(\fN) = \dN \leq r$, then $\bfN \notin \core{\bar\m}$, that is, $\bfN$ is not contained in some minimal reduction of $\bar\m$. 

By \cite[Theorem 4.5]{cpu1}, $\core{\bar\m}$ is in fact the intersection of \emph{general} minimal reductions of $\bar\m$. 
As the generators of $\bar\m$ are all in degree one, its general minimal reductions are precisely the ideals generated by $n-N$ general linear forms. 
Thus, we have established that $\fN \notin (\f,\l)$ for some choice of general linear forms $\l=\ell_1,\ldots,\ell_{n-N}$ whenever $\dN \leq r$. 
A priori, this implies 
\[
	\hf{R/(I,\l)}{\dN} = \hf{R/(\f,\l)}{\dN} - 1, 
\]
where $\hf{M}~$ denotes the Hilbert function of an $R$-module $M$. 
But as $R/(\f,\l)$ is Gorenstein with a 1-dimensional socle in degree $r$, we have in fact 
\[
	\hf{R/(I,\l)}{i} \leq \hf{R/(\f,\l)}{i} - 1 \quad \mbox{for} \quad \dN \leq i \leq r
\]
and consequently 
\begin{align*}
	\lambda(R/(I,\l)) & \; \leq \; \lambda(R/(\f,\l)) - (r-\dN+1).
\end{align*}
As shown in Section~\ref{prelims}, $e(R/I) \leq \lambda(R/(I,\l))$ and we arrive at 
\begin{align*}
	e(R/I) & \ \leq \ \prod_{i=1}^N d_i - \sum_{i=1}^N (d_i-1) + (\dN-1)
\end{align*}
whenever $\dN \leq r$. 

Finally, notice that the condition $\dN \leq r$ is equivalent to $\sum_{i=1}^N (d_i-1) - (\dN-1) \geq 1$. 
If this condition is not satisfied, one may infer by elementary means (see \cite[Lemma 8]{e}) that $e(R/I) \leq e(R/\f) - 1$. 
This proves the inequality 
\[
	e(R/I) \ \leq \ \prod_{i=1}^N d_i - \max\left\{ 1,\ \sum_{i=1}^N (d_i-1) - (\dN-1) \right\}.
\]
\end{proof}

\section*{Acknowledgments}

I thank Craig Huneke, Claudia Polini, and Bernd Ulrich for valuable conversations.

\end{document}